\documentclass[11pt]{article}

\usepackage{amsmath}
\usepackage{amsthm}
\usepackage{amsfonts}
\usepackage{latexsym}
\usepackage{dsfont}
\usepackage{bbm}
\usepackage{color}

\numberwithin{equation}{section}

\theoremstyle{plain}
\newtheorem{Theorem}{Theorem}[section]

\theoremstyle{remark}
\newtheorem{Rem}[Theorem]{Remark}
\theoremstyle{definition}

\DeclareMathOperator{\1}{\mathbbm{1}}

\newcommand{\mn}{\mathbb{N}}
\newcommand{\mr}{\mathbb{R}}
\newcommand{\me}{\mathbb{E}}
\newcommand{\mmp}{\mathbb{P}}

\newcommand{\lix}{\underset{x\to\infty}{\lim}}
\newcommand{\lin}{\underset{n\to\infty}{\lim}}

\title{Functional limit theorems for divergent perpetuities in the contractive case}
\author{Dariusz Buraczewski\footnote{Mathematical Institute, University of
Wroc{\l}aw, Pl. Grunwaldzki 2/4, 50-384 Wroc{\l}aw, Poland
\newline e-mail: dbura@math.uni.wroc.pl} \ \ and \ \
Alexander Iksanov\footnote{Faculty of Cybernetics, Taras
Shevchenko National University of Kyiv, 01601 Kyiv, Ukraine
\newline e-mail: iksan@univ.kiev.ua}}

\begin{document}

\thispagestyle{empty}
\maketitle

\begin{abstract}
Let $\big(M_k, Q_k\big)_{k\in\mn}$ be independent copies of an
$\mr^2$-valued random vector. It is known that if
$Y_n:=Q_1+M_1Q_2+\ldots+M_1\cdot\ldots\cdot M_{n-1}Q_n$ converges
a.s.\ to a random variable $Y$, then the law of $Y$ satisfies the
stochastic fixed-point equation $Y \overset{d}{=} Q_1+M_1Y$, where
$(Q_1, M_1)$ is independent of $Y$. In the present paper we
consider the situation when $|Y_n|$ diverges to $\infty$ in
probability because $|Q_1|$ takes large values with high
probability, whereas the multiplicative random walk with steps
$M_k$'s tends to zero a.s. Under a regular variation assumption we
show that $\log |Y_n|$, properly scaled and normalized, converge
weakly in the Skorokhod space equipped with the $J_1$-topology to
an extremal process. A similar result also holds for the
corresponding Markov chains. Proofs rely upon a deterministic
result which establishes the $J_1$-convergence of certain sums to
a maximal function and subsequent use of the Skorokhod
representation theorem.

\vspace{0,1cm}


\hphantom{2000 Mathematics Subject Classification: }Secondary:
60G50
\end{abstract}

\section{Introduction} \label{sec:Intro_and_main_results}

Let $\big(M_k, Q_k\big)_{k\in\mn}$ be independent copies of a
random vector $\big(M, Q\big)$ with arbitrary dependence of the
components, and let $X_0$ be a random variable which is
independent of $\big(M_k,Q_k\big)_{k\in\mn}$. Then the sequence
$\big(X_n\big)_{n\in\mn_0}$ defined by
\begin{equation}
\label{markov chain}
X_n=M_nX_{n-1}+Q_n, \quad
n\in\mn,
\end{equation}
is a homogeneous Markov chain. In view of the representation
\begin{eqnarray*}
X_n&=&\Psi_n(X_{n-1})=\Psi_n\circ\ldots\circ\Psi_1(X_0)\\&=&
Q_n+M_nQ_{n-1}+\ldots+M_nM_{n-1}\cdot\ldots\cdot M_2 Q_1+
M_nM_{n-1}\cdot\ldots\cdot M_1X_0
\end{eqnarray*}
for $n\in\mn$, where $\Psi_n(t):=Q_n+M_n t$ for $n\in\mn$,
$(X_n)_{n\in\mn}$ is nothing else but the {\it forward} iterated
function system. Closely related is the {\it backward} iterated
function system
$$Y_n:=\Psi_1\circ\ldots\circ\Psi_n(0)=Q_1+M_1Q_2+\ldots+M_1M_2\cdot\ldots M_{n-1}Q_n, \ \ n\in\mn.$$ In the case that $X_0=0$ a.s.\ it is
easily seen that $X_n$ has the same law as $Y_n$ for each fixed
$n$.

Put
$$\Pi_0:=1, \ \ \Pi_n:=M_1M_2\cdot\ldots\cdot M_n, \ \ n\in\mn$$
and assume that
\begin{equation}\label{trivial}
\mmp\{M=0\}=0\quad \text{and}\quad\mmp\{Q=0\}<1
\end{equation}
and
\begin{equation}\label{trivial2}
\mmp\{Q+Mr=r\}<1\quad\text{for all}\quad r\in\mr.
\end{equation}
Then according to Theorem 2.1 in \cite{Goldie+Maller:2000} the
series $\sum_{k\geq 1}\Pi_{k-1}Q_k$ is absolutely a.s.\ convergent
provided that
\begin{equation}\label{33}
\lin \Pi_n=0 \ \ \text{a.s. and} \ \ I:=\int_{(1,\infty)}{\log
x\over A(\log x)}\mmp\{|Q|\in {\rm d}x\}<\infty,
\end{equation}
where $A(x):=\me (\log^-|M|\wedge x)$, $x>0$. The sum $Y$, say, of
the series is then called {\it perpetuity}.

It is also well-known what happens in the 'trivial cases' when at
least one of conditions \eqref{trivial} and \eqref{trivial2} does
not hold.

\noindent (a) If $\mmp\{M=0\}>0$, then $\tau:=\inf\{k\in\mn:
M_k=0\}<\infty$ a.s., and the perpetuity trivially converges, the
limit being an a.s.\ finite random variable $\sum_{k=1}^\tau
\Pi_{k-1}Q_k$. Plainly, its law is a unique invariant measure for
$(X_n)$.

\noindent (b) If $\mmp\{Q=0\}=1$, then $\sum_{k\geq
1}\Pi_{k-1}Q_k=0$ a.s.

\noindent (c) If $\mmp\{Q+Mr=r\}=1$ for some $r\in\mr$, then
either $\delta_r$ is a unique invariant probability measure for
$(X_n)$ or every probability law is an invariant measure, or every
symmetric around $r$ probability law is an invariant measure (see
Theorem 3.1 in \cite{Goldie+Maller:2000} for the details).

Under assumptions \eqref{trivial}, \eqref{trivial2} and \eqref{33}
the Markov chain $\big(X_n\big)$ has a unique invariant
probability measure which is the law of the perpetuity.
Equivalently, the law of $Y$ is a unique solution to the
stochastic fixed-point equation
\begin{equation}
\label{random difference equation}
 Y \overset{d}{=} Q+MY,
\end{equation}
where the vector $(M, Q)$ is assumed independent of $Y$, sometimes
called the {\it random difference equation}. Equations
\eqref{random difference equation} appear in diverse areas of both
applied and pure mathematics and various properties of $Y$ have
attracted considerable attention. Papers
\cite{Alsmeyer+Iksanov+Roesler:2009,Goldie+Maller:2000,Vervaat:1979}
give pointers to relevant literature.

For $(X_n)$ defined by \eqref{markov chain} we write $X_n^v$ to
indicate that $X_0 = v$ for $v\in\mr$. If the first part of
\eqref{33} is in force we infer $|X_n^v-X_n^w| = \Pi_n|v-w|\to 0$
a.s.\ as $n\to\infty$, for any $v,w\in \mr$. Therefore, the case
when $\lin \Pi_n=0$ a.s.\ will be called {\it contractive}.

In the present paper we are interested in the case when conditions
\eqref{trivial}, \eqref{trivial2} and
\begin{equation}\label{30} \lin \Pi_n=0\quad\text{a.s. and}\quad I=\infty
\end{equation}
hold, i.e., the model is still contracting, yet the second
condition in \eqref{33} is violated. By Theorem 2.1 in
\cite{Goldie+Maller:2000} $(Y_n)$ is then a {\it divergent
perpetuity} in the sense that
$|Y_n|=|\sum_{k=1}^n\Pi_{k-1}Q_k|\overset{{\rm P}}{\to}\infty$ as
$n\to\infty$. The purpose of the present paper is to prove
functional limit theorems for the Markov chains $(X_n)$ and for
the divergent perpetuities $(Y_n)$ under the aforementioned
assumptions.

As far as we know Grincevi\v{c}ius \cite{Grincevicius:1975} was
the first to prove a limit theorem for $Y_n$ in the case $\me \log
|M|=0$ under the assumption that $M>0$ a.s. Also, weak convergence
of one-dimensional distributions of divergent perpetuities has
been investigated in \cite{Basu+Roitershtein:2013,
Hitczenko+Wesolowski:2011, Pakes:1983, Rachev+Samorodnitsky:1995}
under various assumptions on $M$ and $Q$. To the best of our
knowledge, (a) functional limit theorems for divergent
perpetuities have not been obtained so far; (b) \cite{Pakes:1983}
is the only contribution to case \eqref{30} which deals with
one-dimensional convergence. We would like to stress that outside
the area of limit theorems we are only aware of two papers
\cite{Kellerer:1992} and \cite{Zeevi+Glynn:2004} which investigate
case \eqref{30}. Unlike \eqref{30} the critical non-contractive
case $\me \log |M|=0$ has received more attention in the
literature, see \cite{Babillot+Bougerol+Elie:1997, Brofferio:2003,
Brofferio+Buraczewski:2014+, Buraczewski:2007, Grincevicius:1975,
Hitczenko+Wesolowski:2011, Rachev+Samorodnitsky:1995}.

Assuming that the tail of $\log^-|M|$ is lighter than that of
$\log^+|Q|$ we state two functional limit theorems thereby
covering a variety of situations. In particular, we do not require
finiteness of $\me \log |M|$. Under \eqref{30} the complementary
case is also possible where the tail of $\log^-|M|$ is not lighter
than that of $\log^+ |Q|$. Take, for instance,
$\mmp\{\log^-|M|>x\}\sim x^{-\alpha}\log x$, $x\to\infty$, and
$\mmp\{\log |Q|\in {\rm d}x\}=\alpha
x^{-\alpha-1}\1_{(1,\infty)}{\rm d}x$ for some $\alpha\in (0,1)$.
Even though this situation is beyond the scope of the present work
we note without going into details that it is unlikely that there
is functional convergence in the Skorokhod space equipped with one
of the standard topologies like $J_1$ or $M_1$. Also, it is worth
to stress that unlike some previous papers on limit theorems for
perpetuities we allow $M$ and $Q$ to take values of both signs.

For $c>0$ and $\alpha>0$, let $N^{(c,\alpha)}:=\sum_k
\varepsilon_{(t_k^{(c,\alpha)},\,j_k^{(c,\alpha)})}$ be a Poisson
random measure on $[0,\infty)\times (0,\infty]$ with mean measure
$\mathbb{LEB}\times \mu_{c,\alpha}$, where $\varepsilon_{(t,\,x)}$
is the probability measure concentrated at $(t,x)\subset
[0,\infty)\times (0,\infty]$, $\mathbb{LEB}$ is the Lebesgue
measure on $[0,\infty)$, and $\mu_{c,\alpha}$ is a measure on
$(0,\infty]$ defined by
$$\mu_{c,\alpha}\big((x,\infty]\big)=cx^{-\alpha}, \ \ x>0.$$
Let $D:=D[0,\infty)$ denote the Skorokhod space of
right-continuous functions defined on $[0,\infty)$ with finite
limits from the left at positive points. Throughout the paper we
use '$\Rightarrow$' to denote weak convergence in the Skorokhod
space $D$ equipped with the $J_1$-topology. We write
'$\Rightarrow$ in $S$' to denote weak convergence in a space $S$
other than $D$. Also, we stipulate hereafter that the supremum
over the empty set is equal to zero.

Theorem \ref{main1} treats the situation in which both $M_k$'s and
$Q_k$'s affect the limit behavior of the processes in question,
whereas in the situation of Theorem \ref{main2} only the
contribution of $Q_k$'s persists in the limit.
\begin{Theorem}\label{main1} Assume that
\begin{equation}\label{mean}
\me \log |M|=-a\in (-\infty, 0),
\end{equation}
that
\begin{equation}\label{2}
\lix x\mmp\{\log |Q|>x\}=c
\end{equation}
for some $c>0$. If
\begin{equation}\label{nondeg}
\mmp\{Y_k=0\}=0
\end{equation}
for each $k\in\mn$, then
\begin{equation}\label{31}
{\log \big|Y_{[n\cdot]+1}\big|\over an} \ \Rightarrow \
\underset{t_k^{(c/a,1)}\leq
\cdot}{\sup}\big(-t_k^{(c/a,1)}+j_k^{(c/a,1)}\big),\quad
n\to\infty,
\end{equation}
and if
\begin{equation}\label{nondeg2}
\mmp\{X_k=0\}=0
\end{equation}
for each $k\in\mn$, then
\begin{equation}\label{32}
{\log \big|X_{[n\cdot]+1}\big|\over an} \ \Rightarrow \
g(\cdot)+\underset{t_k^{(c/a,1)}\leq
\cdot}{\sup}\big(t_k^{(c/a,1)}+j_k^{(c/a,1)}\big), \ \ n\to\infty,
\end{equation}
where $g(t):=-t$, $t\geq 0$.
\end{Theorem}
\begin{Rem}
Conditions \eqref{nondeg} and \eqref{nondeg2} ensure that the
paths of $\log \big|Y_{[n\cdot]+1}|$ and $\log
\big|X_{[n\cdot]+1}|$ belong to $D$. While a simple sufficient
condition for \eqref{nondeg} to hold is continuity of the law of
$Q$, \eqref{nondeg2} holds if either $X_0=0$ a.s. and the law of
$Q$ is continuous or the law of $X_0$ is continuous. Condition
\eqref{nondeg} (\eqref{nondeg2}) is not needed if (a) we replace
$\log$ with $\log^+$ in \eqref{31} (\eqref{32}); (b) consider weak
convergence in $D(0,\infty)$ rather than $D$. The same remark also
concerns Theorem \ref{main2} given below.
\end{Rem}
\begin{Rem}\label{remark1}
Since $X_n\overset{d}{=} Y_n$ for each $n\in\mn$ provided that
$X_0=0$ a.s., the one-dimensional distributions of the limit
processes in \eqref{31} and \eqref{32} must coincide. Moreover,
they can be explicitly computed and are given by
\begin{multline}\label{38}
\mmp\Big\{\underset{t_k^{(c/a,1)}\leq
u}{\sup}\big(-t_k^{(c/a,1)}+j_k^{(c/a,1)}\big)\leq x\Big\} \\
 = \mmp\big\{-u+\underset{t_k^{(c/a,1)}\leq
u}{\sup}(t_k^{(c/a,1)}+j_k^{(c/a,1)})\leq x\big\}=\bigg({x\over
x+u}\bigg)^{c/a}
\end{multline}
for $x\geq 0$ and $u>0$.

Indeed, for $x\geq 0$, the probability on the left-hand side
equals
$$\mmp\big\{N^{(c/a,1)}\big((t,y): t\leq u,
-t+y>x\big)=0\big\}=\exp\big(-\me N^{(c/a,1)}\big((t,y): t\leq u,
-t+y>x\big)\big)$$ because $N^{(c/a,1)}\big((t,y): t\leq u,
-t+y>x\big)$ is a Poisson random variable. It remains to note that
\begin{eqnarray*}
\me N^{(c/a,1)}\big((t,y): t\leq u,
-t+y>x\big)&=&\int_0^u\int_{[0,\infty)}\1_{\{-t+y>x\}}\mu_{c/a,\,1}({\rm
d}y){\rm d}t\\&=&(c/a)\int_0^u(x+t)^{-1}{\rm
d}t\\&=&(c/a)(\log(x+u)-\log x).
\end{eqnarray*}
\end{Rem}
\begin{Rem}
Theorem 5(ii) in \cite{Pakes:1983} states that, for fixed $a>0$,
\begin{equation}\label{36}
\lin \mmp\bigg\{\log\bigg(\sum_{k=0}^n e^{-ak}|Q_{k+1}|\bigg)\leq
anx\bigg\}=\bigg({x\over x+1}\bigg)^{c/a}, \ \ x\geq 0
\end{equation}
provided that
\begin{equation*}
\lix x\big(1-\me \exp(-e^{-x}|Q|)\big)=c\in (0,\infty).
\end{equation*}
By an Abelian-Tauberian argument the last relation is equivalent
to \eqref{2}. This implies that convergence \eqref{36} follows
from \eqref{31} and \eqref{38}.
\end{Rem}
\begin{Theorem}\label{main2}
Suppose that $\mmp\{M=0\}=0$, $\lin \Pi_n=0$ a.s., and that
\begin{equation}\label{20}
\mmp\{\log |Q|>x\}\sim x^{-\alpha}\ell(x), \ \ x\to\infty
\end{equation}
for some $\alpha\in (0,1]$ and some $\ell$ slowly varying at
$\infty$. Let $(b_n)$ be a positive sequence which satisfy $\lin
n\mmp\{\log |Q|>b_n\}=1$. In the case $\alpha=1$ assume
additionally\footnote{Among other things this implies
$\me\log^+|Q|=\infty$.} that $\lix \ell(x)=+\infty$. In the case
$\me \log^- |M|=\infty$ assume that
\begin{equation}\label{18}
\lix {\me \big(\log^-|M|\wedge x\big)\over x\mmp\{\log|Q|>x\}}=0.
\end{equation}
If condition \eqref{nondeg} holds, then
\begin{equation}\label{35}
{\log \big|Y_{[n\cdot]+1}\big| \over b_n} \ \Rightarrow \
\underset{t_k^{(1,\,\alpha)}\leq
\cdot}{\sup}\,j_k^{(1,\,\alpha)},\quad n\to\infty,
\end{equation}
and if condition \eqref{nondeg2} holds, then
\begin{equation}\label{355}
{\log \big|X_{[n\cdot]+1}\big|\over b_n} \ \Rightarrow \
\underset{t_k^{(1,\,\alpha)}\leq
\cdot}{\sup}\,j_k^{(1,\,\alpha)},\quad n\to\infty.
\end{equation}
\end{Theorem}

\begin{Rem}
Theorem 5(iii) in \cite{Pakes:1983} states that, for fixed $a>0$,
\begin{equation}\label{39}
\lin \mmp\bigg\{\log\bigg(\sum_{k=0}^n e^{-ak}|Q_{k+1}|\bigg)\leq
b_n x\bigg\}=\exp(-x^{-\alpha}), \ \ x\geq 0
\end{equation}
provided that the function $x\mapsto 1-\me \exp(-e^{-x}|Q|)$ is
regularly varying at $\infty$ with index $-\alpha$, $\alpha\in
(0,1)$, and $(b_n)$ satisfies $n\big(1-\me
\exp(-e^{-b_n}|Q|)\big)=1$. By an Abelian theorem, $$1-\me
\exp(-e^{-x}|Q|) \ \sim \ \mmp\{\log |Q|>x\}, \ \ x\to\infty.$$
Therefore, \eqref{39} follows from \eqref{35} after noting that
\begin{equation}\label{43}
\mmp\big\{\underset{t_k^{(1,\,\alpha)}\leq
u}{\sup}\,j_k^{(1,\,\alpha)}\leq
x\big\}=\mmp\big\{N^{(1,\alpha)}\big((t,y): t\leq u,
y>x\big)=0\big\}=\exp(-ux^{-\alpha}), \ \ x\geq 0
\end{equation}
for each $u>0$.
\end{Rem}

The rest of the paper is structured as follows. In Section
\ref{aux1} we state and prove Theorem \ref{aux}, a deterministic
result which is our key tool for dealing with the functional limit
theorems. With this at hand, Theorem \ref{main1} and Theorem
\ref{main2} are then proved in Section \ref{sect3} and Section
\ref{sect4}, respectively.

\section{Main technical tool}\label{aux1}

Denote by $M_p$ the set of Radon point measures $\nu$ on
$[0,\infty)\times (0,\infty]$ which satisfy
\begin{equation}\label{1}
\nu([0,T]\times [\delta,\infty])<\infty
\end{equation}
for all $\delta>0$ and all $T>0$. The $M_p$ is endowed with the
vague topology. Denote by $M_p^\ast$ the set of $\nu\in M_p$ which
satisfy $$\nu([0,T]\times (0,\infty])<\infty$$ for all $T>0$.
Define the mapping $G$ from $D\times M_p$ to $D$
by\footnote{Assumption \eqref{1} ensures that $G(f,\nu)\in D$. If
\eqref{1} does not hold, $G(f,\nu)$ may lost right-continuity.}
\begin{equation*}
G\left(f,\nu\right)(t):=
\begin{cases}
        \underset{k:\ \tau_k\leq t}{\sup}(f(\tau_k)+y_k),  & \text{if} \ \tau_k\leq t \ \text{for
some} \ k,\\
        f(0), & \text{otherwise},
\end{cases}
\end{equation*}
where $\nu = \sum_k \varepsilon_{(\tau_k,\,y_k)}$. Also, for each
$n\in\mn$, we define the mapping $F_n$ from $D\times M_p^\ast$ to
$D$ by
\begin{equation*}
F_n\left(f,\nu\right)(t):=
\begin{cases}
c_n^{-1}\log^+ \big|\sum_{k:\,\tau_k\leq t}\pm \exp
(c_n(f(\tau_k)+y_k))\big|, & \text{if} \ \tau_k\leq t \ \text{for
some} \ k,\\
        f^+(0), & \text{otherwise},
\end{cases}
\end{equation*}
where the signs $+$ and $-$ are arbitrarily arranged, and $(c_n)$
is some sequence of positive numbers. The definition of $F_n$ in the case of empty sum stems from the fact that we define $\big|\sum_{k:\,\tau_k\leq t}\pm \exp
(c_n(f(\tau_k)+y_k))\big|:=\exp (c_n f(0))$ if there is no $k$ such that $\tau_k\leq t$.

\begin{Theorem}\label{aux}
For $n\in\mn$, let $f_n\in D$ and $\nu_n\in M_p$. Let
$\big(\tau^{(n)}_k, y^{(n)}_k\big)$ be the points of $\nu_n$,
i.e., $\nu_n=\sum_k \varepsilon_{(\tau^{(n)}_k,\,y^{(n)}_k)}$.
Assume that $f_0$ is continuous with $f_0(0)=0$ and
\begin{itemize}
\item[(A1)] $\nu_0(\{0\}\times (0,\infty])=0$ and $\nu_0((r_1,r_2)\times(0,\infty])\geq 1$
for all positive $r_1$ and $r_2$ such that $r_1<r_2$;\\
\item[(A2)] $\nu_0 = \sum_k \varepsilon_{\big(\tau^{(0)}_k,\,y^{(0)}_k\big)}$ does not have clustered jumps, i.e.,
$\tau^{(0)}_k\neq \tau^{(0)}_j$ for $k\neq j$;\\
\item[(A3)] if not
all the signs under the sum defining $F_n$ are the same, then
\begin{equation}
\label{eq:a3-1}
f_0(\tau^{(0)}_i)+y^{(0)}_i\neq f_0(\tau^{(0)}_j)+y^{(0)}_j\ \ \mbox{for
$i\neq j$}
\end{equation}
and
\begin{equation}
\label{eq:a3-2}
\sup_{\tau^{(0)}_k\leq T,\,y^{(0)}_k\leq
\gamma}\big(f_0(\tau^{(0)}_k)+y^{(0)}_k\big)>0
\end{equation}
for each $T>0$ such that
$\nu_0(\{T\}, (0,\infty])=0$ and small enough $\gamma>0$;
\item[(A4)] $\lin c_n=\infty$ and
\begin{equation}\label{sequ}
\lin c_n^{-1}\log \#\{k: \tau_k^{(n)}\leq T\}=0
\end{equation}
for each $T>0$ such that $\nu_0(\{T\}, (0,\infty])=0$;
\item[(A5)] $\lin f_n= f_0$ in $D$ in the $J_1$-topology.  
\item[(A6)] $\lin {\nu_n }=\nu_0$ in  $M_p$.
\end{itemize}
Then
\begin{equation}\label{1.3}
\lin F_n(f_n,\nu_n)= G(f_0,\nu_0)
\end{equation}
in $D$ in the $J_1$-topology.
\end{Theorem}

\begin{proof}
It suffices to prove convergence \eqref{1.3} in $D[0,T]$ for any
$T>0$ such that $\nu_0(\{T\}\times(0,\infty])=0$ because the last
condition ensures that $F(f_0,\nu_0)$ is continuous at $T$.

If all the signs under the sum defining $F_n$ are the same, then
$$G(f_n,\nu_n)(t)\leq F_n(f_n,\nu_n)(t)\leq c_n^{-1}\log^+ \#\{k: \tau_k^{(n)}\leq t\}+G(f_n,\nu_n)(t)$$ for all $t\in [0,T]$.
In this case, \eqref{1.3} is a trivial consequence of Theorem 1.3
in \cite{Iksanov+Pilipenko:2014} which treats the convergence
$\lin G(f_n,\nu_n)= G(f_0,\nu_0)$ in $D$.

In what follows we thus assume that not all the signs are the
same. Let $\rho=\{0=s_0<s_1<\dots <s_m=T\}$ be a partition of
$[0,T]$ such that
\begin{equation*}
\nu_0(\{s_k\}\times (0,\infty])=0, \ \ k=1,...,m.
\end{equation*}
Pick now $\gamma>0$ so small that
\begin{equation}\label{2.1'}
\nu_0((s_k,s_{k+1})\times (\gamma,\infty])\geq 1, \ \ k=0,...,m-1
\end{equation}
and that $\sup_{\tau^{(0)}_k\leq
T,\,y^{(0)}_k>\gamma}(f_0(\tau^{(0)}_k)+y^{(0)}_k)>0$. The latter
is possible because\newline $\sup_{\tau^{(0)}_k\leq
T}(f_0(\tau^{(0)}_k)+y^{(0)}_k)>0$ as a consequence of
\eqref{eq:a3-2}.

Condition (A6) implies that $\nu_0([0,T]\times
(\gamma,\infty])=\nu_n([0,T]\times (\gamma,\infty])=p$ for large
enough $n$ and some $p\geq 1$. Denote by
$(\bar{\tau}_i,\bar{y}_i)_{1\leq i\leq p}$ an enumeration of the
points of $\nu_0$ in $[0,T]\times (\gamma,\infty]$ with
$\bar{\tau}_1<\bar{\tau}_2<\ldots<\bar{\tau}_p$ and by
$(\bar{\tau}_i^{(n)}, \bar{y}_i^{(n)})_{1\leq i\leq p}$ the
analogous enumeration of the points of $\nu_n$ in $[0,T]\times
(\gamma,\infty]$. Then
\begin{equation*}
\lim_{n\to\infty}\sum_{i=1}^p(| \bar{\tau}^{(n)}_i- \bar{\tau}_i
|+|\bar{y}^{(n)}_i-\bar{y}_i|)=0
\end{equation*}
and more importantly
\begin{equation}\label{2.2}
\lim_{n\to\infty}\sum_{i=1}^p(| f_n(\bar{\tau}^{(n)}_i)-
f_0(\bar{\tau}_i)|+|\bar{y}^{(n)}_i-\bar{y}_i|)=0
\end{equation}
because (A5) and the continuity of $f_0$ imply that $\lin f_n=
f_0$ uniformly on $[0,T]$.

Define $\lambda_n$ to be continuous and strictly increasing
functions on $[0,T]$ with $\lambda_n(0) =0$, $\lambda_n(T) =T$,
$\lambda_n(\bar{\tau}^{(n)}_i)=\bar{\tau}_i$ for $i=1,\ldots,p$,
and let $\lambda_n$ be linearly interpolated elsewhere on $[0,T]$.
For $n\in\mn$ and $t\in [0,T]$, set
\begin{equation*}
V_n(t) := \sum_{\bar{\tau}_i=\lambda_n(\bar{\tau}^{(n)}_i)\leq
t}\pm \exp\big(c_n(f_n(\bar{\tau}^{(n)}_i)+\bar{y}^{(n)}_i)\big)
\end{equation*}
and \begin{equation*} W_n(t) := \sum_{\lambda_n(\tau^{(n)}_k)\leq
t}\pm \exp\big(c_n (f_n(\tau^{(n)}_k)+y^{(n)}_k)\big) - V_n(t).
\end{equation*}
With this at hand we have
\begin{eqnarray}\label{principal}
d_T(F_n(f_n,\nu_n), G(f_0,\nu_0)) &\leq&
\sup_{t\in[0,\,T]}|\lambda_n(t)-t|\\&+&c_n^{-1}\sup_{t\in[0,\,T]}\Big|\log^+\big|
W_n(t) + V_n(t) \big| - \log^+\big| V_n(t) \big|\Big|\notag\\&+&
\sup_{t\in[0,\,T]}\Big|c_n^{-1}\log^+\big|V_n(t)\big| -
\sup_{\bar{\tau}_i\leq t}(f_0(\bar{\tau}_i)+\bar{y}_i)
\Big|\notag\\&+& \sup_{t\in[0,T]}\Big|\sup_{\bar{\tau}_i\leq
t}\big(f_0(\bar{\tau}_i)+\bar{y}_i\big)-\sup_{\tau^{(0)}_k\leq
t}\big(f_0(\tau^{(0)}_k)+y^{(0)}_k\big )\Big|,\notag
\end{eqnarray}
where $d_T$ is the standard Skorokhod metric on $D[0,T]$.

We treat the terms on the right-hand side of \eqref{principal}
separately.

\noindent {\sc 1st term}. The relation $\lin
\sup_{t\in[0,\,T]}|\lambda_n(t)-t|=0$ is easily checked.

\noindent {\sc 2nd term}. We denote the second term by $I_n(\gamma)$ and use inequality
$$|\log^+|x|-\log^+|y||\leq \log(1+|x-y|),\quad x,y\in\mr$$ which yields
\begin{eqnarray}
I_n(\gamma)&\leq& c_n^{-1}\sup_{t\in
[0,T]}\log\big(1+\big|W_n(t)\big|\big)\notag\\&\leq& c_n^{-1}
\log\bigg(1+ \sum_{\lambda_n(\tau^{(n)}_k)\leq
T,\,\tau^{(n)}_k\neq \bar{\tau}^{(n)}_i} \exp\big(c_n
(f_n(\tau^{(n)}_k)+y^{(n)}_k)\big)\bigg)\notag\\&\leq& c_n^{-1}\log
\bigg(1+\#\big\{k: \tau^{(n)}_k\leq T,\,\tau^{(n)}_k\neq
\bar{\tau}^{(n)}_i
\big\}\notag\\&\times&\sup_{\tau^{(n)}_k\leq
T,\,\tau^{(n)}_k\neq \bar{\tau}^{(n)}_i}\exp\big(c_n
(f_n(\tau^{(n)}_k)+y^{(n)}_k)\big)\bigg)\notag\\&\leq& c_n^{-1}\log
\#\big\{k: \tau^{(n)}_k \leq T\big\}+ \sup_{\tau^{(n)}_k\leq
T,\,\tau^{(n)}_k\neq
\bar{\tau}^{(n)}_i}\big(f_n(\tau^{(n)}_k)+y^{(n)}_k\big)\notag\\&+&
\bigg(c_n \#\big\{k: \tau^{(n)}_k\leq T,\,\tau^{(n)}_k\neq
\bar{\tau}^{(n)}_i\big\}\bigg)^{-1}\notag
\\&\times&\exp\bigg(-c_n\sup_{\tau^{(n)}_k\leq
T,\,\tau^{(n)}_k\neq \bar{\tau}^{(n)}_i}
\big(f_n(\tau^{(n)}_k)+y^{(n)}_k\big)\bigg)\label{tech}
\end{eqnarray}
having utilized $\log (1+x)\leq \log x+1/x$, $x>0$ and that
$\lambda_n(\tau^{(n)}_k)\leq T$ iff $\tau^{(n)}_k \leq T$. The
first term on the right-hand side of \eqref{tech} converges to
zero in view of \eqref{sequ}. As to the second, we apply Theorem
1.3 in \cite{Iksanov+Pilipenko:2014} to infer
\begin{eqnarray}\label{princ}
\sup_{\tau^{(n)}_k\leq T,\,\tau^{(n)}_k\neq
\bar{\tau}^{(n)}_i}\big(f_n(\tau^{(n)}_k)+y^{(n)}_k\big)
&=&\sup_{\tau^{(n)}_k\leq T,\,y^{(n)}_k\leq \gamma}(f_n(\tau^{(n)}_k)+y^{(n)}_k)\notag\\
&\to& \sup_{\tau^{(0)}_k\leq T,\,y^{(0)}_k\leq
\gamma}\big(f_0(\tau^{(0)}_k)+y^{(0)}_k\big),
\end{eqnarray}
as $n\to\infty$. The latter goes to zero as $\gamma\to 0$ because
$f_0=0$ by assumption. Finally, the last term on the right-hand
side of \eqref{tech} tends to zero as $n\to\infty$ for the
principal factor of exponential growth does so as a consequence of
\eqref{princ} and the assumption $\sup_{\tau^{(0)}_k\leq
T,\,y^{(0)}_k\leq \gamma}(f_0(\tau^{(0)}_k)+y^{(0)}_k)>0$.
Summarizing we have proved that $\underset{\gamma\to
0}{\lim}\underset{n\to\infty}{\lim\sup}\,I_n(\gamma)=0$.

\noindent {\sc 3rd term}. Denote the third term of
\eqref{principal} by $J_n$. We shall use the inequality
$$J_n\leq \sup_{t\in [0,\,T]}A_n(t)+c_n^{-1}\sup_{t\in[0,\,T]}\log^-| V_n(t)|,$$ where
$A_n(t):=\Big|c_n^{-1}\log| V_n(t)|-\sup_{\bar{\tau}_i\leq
t}(f_0(\bar{\tau}_i)+\bar{y}_i) \Big|$, $t\in [0,T]$.

If $t\in [0,\bar{\tau}_1)$, then $A_n(t)=|f_n(0)-f_0(0)|\to 0$ as
$n\to\infty$ by the definition of the functionals. Let now $t\in
[\bar{\tau}_k,\bar{\tau}_{k+1})$, $k=1,\ldots,p-1$ or $t\in
[\bar{\tau}_p, T]$. Since all
$\exp(f_0(\bar{\tau}_1)+\bar{y}_1),\ldots,
\exp(f_0(\bar{\tau}_k)+\bar{y}_k)$ are distinct by \eqref{eq:a3-1}
and $$\lin \exp(f_n(\bar{\tau}^{(n)}_j)+\bar{y}^{(n)}_j)
=\exp(f_0(\bar{\tau}_j)+\bar{y}_j), \quad j=1,\ldots,k$$ by
\eqref{2.2}, we conclude that
$\exp(f_n(\bar{\tau}_1^{(n)})+\bar{y}^{(n)}_1),\ldots,
\exp(f_n(\bar{\tau}_k^{(n)})+\bar{y}_k^{(n)})$ are all distinct,
for large enough $n$. Denote by $a_{k,n} < \ldots < a_{1,n}$ their
increasing rearrangement\footnote{Although $a_{j,n}$'s depend on
$t$ we suppress this dependence for the sake of clarity.} and put
\begin{equation*}
B_n(t):=c_n^{-1}\log \bigg|1\pm \bigg({a_{2,n}\over
a_{1,n}}\bigg)^{c_n}\pm\ldots\pm \bigg({a_{k,n}\over
a_{1,n}}\bigg)^{c_n}\bigg|.
\end{equation*}
Since $\lin \bigg( \pm \big({a_{2,n}\over
a_{1,n}}\big)^{c_n}\pm\ldots\pm \big({a_{k,n}\over
a_{1,n}}\big)^{c_n}\bigg)=0$, there is an $N_k$ such that
$$
|B_n(t)| \le c_n^{-1} \qquad \mbox{ for } n\ge N_k.
$$
Summarizing we have
\begin{equation}\label{imp}
\sup_{t\in [0,\,T]}|B_n(t)|\leq c_n^{-1}\quad\text{for all} \
n\geq \max (N_1,\ldots, N_p).
\end{equation}
With these at hand we can proceed as follows
\begin{eqnarray*}
A_n(t)&=&\Big|\sup_{\bar{\tau}_i\leq
t}\big(f_n(\bar{\tau}^{(n)}_i)+\bar{y}^{(n)}_i\big)+B_n(t)-\sup_{\bar{\tau}_i\leq
t}\big(f_0(\bar{\tau}_i)+\bar{y}_i\big)\Big|\\
&\leq& \Big|\sup_{\bar{\tau}_i\leq
t}\big(f_n(\bar{\tau}^{(n)}_i)+\bar{y}^{(n)}_i\big)-\sup_{\bar{\tau}_i\leq
t}\big(f_0(\bar{\tau}_i)+\bar{y}_i\big)|+|B_n(t)\Big|\\
&\leq& \sum_{i=1}^p\Big(\big|
f_n(\bar{\tau}^{(n)}_i)-
f_0(\bar{\tau}_i)\big|+\big|\bar{y}^{(n)}_i-\bar{y}_i\big|\Big)+|B_n(t)|.
\end{eqnarray*}
In view of \eqref{2.2} and \eqref{imp} the right-hand side tends
to zero uniformly in $t\in [0,T]$ as $n\to\infty$.

We already know that $$\lin \sup_{t\in[0,\,T]} c_n^{-1}\log\big|
V_n(t)\big| =\sup_{\bar{\tau}_i\leq
T}(f_0(\bar{\tau}_i)+\bar{y}_i).$$ Recalling that
$$\sup_{\bar{\tau}_i\leq
T}\big(f_0(\bar{\tau}_i)+\bar{y}_i\big)=\sup_{\tau^{(0)}_k\leq
T,\,y^{(0)}_k>\gamma}\big(f_0(\tau^{(0)}_k)+y^{(0)}_k\big)>0$$ we
infer $\lin \sup_{t\in[0,\,T]}\big| V_n(t)\big|=+\infty$ and
thereupon $\sup_{t\in[0,\,T]}\log^-\big| V_n(t)\big|=0$ for large
enough $n$. Hence $\lin J_n=0$.

\noindent {\sc 4th term}. In the proof of Theorem 1.3 in
\cite{Iksanov+Pilipenko:2014} it is shown that\footnote{Condition
\eqref{2.1'} is only used in this part of the proof.}
$$\sup_{t\in[0,T]}|\sup_{\bar{\tau}_i\leq
t}(f_0(\bar{\tau}_i)+\bar{y}_i)-\sup_{\tau^{(0)}_k\leq
t}(f_0(\tau^{(0)}_k)+y^{(0)}_k)|\leq
\omega_{f_0}(2|\rho|)+\gamma,$$ where
$|\rho|:=\max_i(s_{i+1}-s_i)$ and
$\omega_{f_0}(\varepsilon):=\underset{|u-v|<\varepsilon,\,u,v\geq
0}{\sup}\,|f_0(u)-f_0(v)|$ is the modulus of continuity of $f_0$.
Of course, the right-hand side of the last inequality tends to
zero on sending $|\rho|$ and $\gamma$ to zero.

Collecting pieces together and letting in \eqref{principal}
$n\to\infty$ and then $|\rho|$ and $\gamma$ tend to zero we arrive
at the desired conclusion $$\lin d_T(F_n(f_n,\nu_n),
G(f_0,\nu_0))=0.$$
\end{proof}

\section{Proof of Theorem \ref{main1}}\label{sect3}

\noindent {\sc Proof of \eqref{31}}. We first show that
\begin{equation}\label{negative}
{\log^-|Y_{[n\cdot]+1}|\over an} \ \Rightarrow \ h(\cdot),
\end{equation}
where $h(t)=0$, $t\geq 0$. To this end, we intend to check that
conditions \eqref{trivial}, \eqref{trivial2} and \eqref{30} hold.
If they do, then, as $n\to\infty$, $|Y_n|\overset{{\rm P}}{\to}
\infty$ by Theorem 2.1 in \cite{Goldie+Maller:2000} and thereupon
$\sup_{t\in [0,\,T]}|Y_{[nt]+1}|=\sup_{1\leq k\leq
[nT]+1}|Y_k|\overset{{\rm P}}{\to}\infty$ for each $T>0$. This
entails $\sup_{t\in [0,\,T]}\log^-|Y_{[nt]+1}|=0$ for each $T>0$
and large enough $n$ which proves \eqref{negative}. Assumption
\eqref{mean} entails $\lin \Pi_n=0$ a.s. and $\mmp\{M=0\}=0$.
Condition $\mmp\{Q=0\}=0$ is a part of \eqref{nondeg}. Suppose
$Q+Mr=r$ a.s.\ for some $r\in\mr$. In view of $\mmp\{Q=0\}=0$ we
have $r\neq 0$ and then $|Q|/|r|=|1-M|\leq 1+|M|$ a.s. Since $\me
\log (1+|M|)<\infty$ by \eqref{mean} we must have
$\me\log^+|Q|<\infty$. This contradiction completes the proof of
\eqref{negative}.

For $k\in\mn_0$, set $S_k:=\log |\Pi_k|$ and $\eta_{k+1}:=\log
|Q_{k+1}|$. As a consequence of the strong law of large numbers,
\begin{equation}\label{6}
{S_{[n\cdot]}\over an} \Rightarrow g(\cdot), \ \ n\to\infty,
\end{equation}
where $g(t):=-t$, $t\geq 0$ (actually, in \eqref{6} the a.s.\
convergence holds, see Theorem 4 in \cite{Glynn+Whitt:1988}).
According to Corollary 4.19 (ii) in \cite{Resnick:1987} condition
\eqref{2} entails
\begin{equation}\label{7}
\sum_{k\geq 0}\1_{\{\eta_{k+1}>0\}}\varepsilon_{(n^{-1}k,\,
(an)^{-1}\eta_{k+1})} \ \Rightarrow \ N^{(c/a,1)}, \ \ n\to\infty
\end{equation}
in $M_p$, see Section \ref{aux1} for the definition of $M_p$. Now
relations \eqref{6} and \eqref{7} can be combined into the joint
convergence
\begin{equation*}
\bigg((an)^{-1}S_{[n\cdot]},\sum_{k\geq
0}\1_{\{\eta_{k+1}>0\}}\varepsilon_{(n^{-1}k,\,
(an)^{-1}\eta_{k+1})}\bigg) \ \Rightarrow \ \big(g(\cdot),
N^{(c/a,1)}\big) \ \text{as} \ \ n\to\infty
\end{equation*}
in $D[0,\infty)\times M_p$. By the Skorokhod representation
theorem there are versions which converge a.s. Retaining the
original notation for these versions we want to apply Theorem
\ref{aux} with $f_n(\cdot)=(an)^{-1}S_{[n\cdot]}$, $f_0=g$,
$\nu_n=\sum_{k\geq 0} \1_{\{\eta_{k+1}>0\}}
\varepsilon_{\{n^{-1}k,\, (an)^{-1}\eta_{k+1}\}}$, $\nu_0=
N^{(c/a,1)}$, $c_n=an$ and the signs $\pm$ defined by ${\rm
sgn}(\Pi_k Q_{k+1})$ to conclude
$${\log^+|Y_{[n\cdot]+1}|\over an} \ \Rightarrow \
\underset{t_k^{(c/a,1)}\leq
\cdot}{\sup}\big(-t_k^{(c/a,1)}+j_k^{(c/a,1)}\big).$$ Of course,
this together with \eqref{negative} proves \eqref{31}.

Thus it remains to check that all the assumptions of Theorem
\ref{aux} hold. We already know that conditions  (A5) and (A6) are
fulfilled. Condition \eqref{sequ} holds trivially. Further
$N^{(c/a,1)}([0,T]\times [\delta,\infty])<\infty$ a.s.\ for all
$\delta>0$ and all $T>0$ because $ \mu_{c/a,1}([\delta,\infty])
<\infty$. Plainly, $N^{(c/a,1)}(\{0\}\times (0,+\infty])=0$ a.s.,
and $N^{(c/a,1)}((r_1,r_2)\times(0,\infty])\geq 1$ a.s.\ whenever
$0<r_1<r_2$ because  $\mu_{c/a,1}((0,\infty])=\infty$. This gives
(A1).

Next we check \eqref{eq:a3-1}. Our argument is similar to that
given on p.~223 in \cite{Resnick:2007}. We fix any $T>0$,
$\delta>0$ and use the representation
$$N^{(c/a,1)}([0,T]\times (\delta,\infty]\cap \cdot)=
\sum_{k=1}^N\varepsilon_{(U_k,V_k)}(\cdot),$$ where $(U_i)$ are
i.i.d. with the uniform $[0,T]$ distribution, $(V_j)$ are iid with
$\mmp\{V_1\leq x\}=(1-\delta/x)\1_{(\delta,\infty)}(x)$, and $N$
has the Poisson distribution with parameter $Tc/(a\delta)$, all
the random variables being independent. It suffices to prove that
$$I:=\mmp\{N\geq 2, -U_k+V_k=-U_i+V_i\quad\text{for some} \ 1\leq k<j\leq
N\}=0.$$ This is a consequence of the fact that $-U_1+V_1$ has a
continuous distribution which implies
$\mmp\{-U_1+V_1=-U_2+V_2\}=0$. Indeed,
\begin{eqnarray*}
I&=&\sum_{n\geq 2}\mmp\{-U_k+V_k=-U_i+V_i\quad\text{for some} \
1\leq k<j\leq n\}\mmp\{N=n\}\\&=& \sum_{n\geq 2}{n\choose
2}\mmp\{-U_1+V_1=-U_2+V_2\}\mmp\{N=n\}=0
\end{eqnarray*}
An analogous working leads to the conclusion that $N^{(c/a,1)}$
does not have clustered jumps a.s., i.e., (A2) holds. The last
thing that needs to be checked is condition \eqref{eq:a3-2}.
Arguing as in Remark \ref{remark1} we infer
\begin{eqnarray*}
&&\mmp\bigg\{\underset{t_k^{(c/a,1)}\leq T,\,j_k^{(c/a,\,1)}\leq
\gamma}{\sup}(-t_k^{(c/a,1)}+j_k^{(c/a,1)})\leq
0\bigg\}\\&=&\exp\bigg(-\me N^{(c/a,1)}\big((t,y): t\leq T, y\leq
\gamma,
y>t\big)\bigg)\\&=&\exp\bigg(-(c/a)\int_0^\gamma(t^{-1}-\gamma^{-1}){\rm
d}t\bigg)=0
\end{eqnarray*}
for any $T>0$ and any $\gamma\in (0,T)$.

\noindent {\sc Proof of \eqref{32}}. Without loss of generality we
assume that $X_0=0$ a.s. and use the representation
\begin{equation}\label{40}
X_{[n\cdot]+1}=\Pi_{[n\cdot]+1}\sum_{k=0}^{[n\cdot]}\Pi^\ast_kQ^\ast_{k+1},
\end{equation}
where $\Pi_k^\ast:=\Pi_k^{-1}$, $k\in\mn_0$ and
$Q^\ast_k:=Q_k/M_k$ (with generic copy $Q^\ast$), $k\in\mn$.

Observe that
$${\underset{0\leq t\leq T}{\sup}\,|S_{[nt]+1}-S_{[nt]}|\over n}={\underset{1\leq k\leq [nT]+1}{\max}\,\big|\log |M_k|\big|\over n}
\ \overset{P}{\to} \ 0, \ \ n \to\infty$$ for every $T>0$, because
$\lix x\mmp\big\{\big|\log |M| \big|>x\big\}=0$ as a consequence
of $\me |\log |M||<\infty$. This together with \eqref{6} proves
\begin{equation}\label{66}
{\log |\Pi_{[n\cdot]+1}|\over an} \Rightarrow g(\cdot),\quad
n\to\infty,
\end{equation}
where $g(t)=-t$, $t\geq 0$. Further, write, for $\varepsilon\in
(0,1)$ and $x>0$,
\begin{eqnarray}
\mmp\{\log |Q|>(1+\varepsilon)x\}-\mmp\{\log |M|>\varepsilon
x\}&\leq& \mmp\{\log |Q|-\log |M|>x\}\notag\\&\leq& \mmp\{\log
|Q|>(1-\varepsilon)x\}\notag\\&+&\mmp\{\log^-|M|>\varepsilon
x\}.\label{41}
\end{eqnarray}
Multiplying the inequality by $x$, sending $x\to\infty$ and then
$\varepsilon\to 0$ yields
$$\mmp\{\log |Q^\ast|>x\}=\mmp\{\log |Q|-\log
|M|>x\} \ \sim \ \mmp\{\log |Q|>x\} \ \sim \ cx^{-1}, \ \
x\to\infty.$$

Set $M^\ast:=1/M$. Conditions \eqref{trivial} and \eqref{trivial2}
with $(M,Q)$ replaced by $(M^\ast, Q^\ast)$ are easily checked.
Also, we have $\lin \Pi^\ast_n=\infty$ a.s. Hence $|\sum_{k=1}^n
\Pi^\ast_{k-1}Q^\ast_k|\overset{{\rm P}}{\to}\infty$ as
$n\to\infty$ by Theorem 2.1 in \cite{Goldie+Maller:2000}. Arguing
in the same way as in the proof of \eqref{31} we see that
$${\log^-|\sum_{k=0}^{[n\cdot]}\Pi^\ast_kQ^\ast_{k+1}|\over an} \ \Rightarrow \
h(\cdot),\quad n\to\infty.$$ An application of Theorem \ref{aux}
gives\footnote{We omit details which are very similar to but
simpler than those appearing in the proof of \eqref{31}.}
$${\log^+|\sum_{k=0}^{[n\cdot]}\Pi^\ast_kQ^\ast_{k+1}|\over an} \
\Rightarrow \ \underset{t_k^{(c/a,\,1)}\leq
\cdot}{\sup}\big(t_k^{(c/a,1)}+j_k^{(c/a,1)}\big),\quad
n\to\infty.$$ Now \eqref{32} follows by a combination of the last
two relations and \eqref{66}.

\section{Proof of Theorem \ref{main2}}\label{sect4}

The proof proceeds along the lines of that of Theorem \ref{main1}
but is simpler for the contribution of $M_k$'s is negligible.
Therefore we only provide details for fragments which differ
principally from the corresponding ones in the proof of Theorem
\ref{main1}.

Observe that
\begin{equation}\label{19}
\lin {b_n\over n}=+\infty.
\end{equation}
Indeed, since $(b_n)$ is a regularly varying sequence of index
$1/\alpha$, this is trivial when $\alpha\in (0,1)$. If $\alpha=1$,
this follows from the relation $b_n/n\sim \ell(b_n)$ as
$n\to\infty$ and our assumption that $\lix \ell(x)=\infty$.

\noindent {\sc Proof of \eqref{35}}. As far as
\begin{equation}\label{negative2}
{\log^-|Y_{[n\cdot]+1}|\over b_n} \ \Rightarrow \ h(\cdot),\quad
n\to\infty
\end{equation}
is concerned which is the counterpart of \eqref{negative} we have
to check two things that are not obvious in the case when $\me
\log^- |M|=\infty$: condition \eqref{trivial2} and
$I=\int_{(1,\infty)}{\log x\over A(\log x)}\mmp\{|Q|\in {\rm
d}x\}=\infty$.

Assume first that $\mmp\{Q+Mr=r\}=1$ for some $r\neq 0$. In view
of $|Q-r|=|M||r|$, the tails of $\log^+ |Q|$ and $\log^+ |M|$ must
exhibit the same asymptotics. However, this is not a case, for the
tail of $\log^+ |Q|$ is heavier than that of $\log^+ |M|$.

Next, according to \eqref{18}, for any $B>0$ there exists $x_0>0$
such that $${\log x\over A(\log x)}\geq {B\over \mmp\{|Q|>x\}}$$
whenever $x\geq x_0$. Hence,
$$I\geq B\int_{[x_0,\infty)}{\mmp\{|Q|\in {\rm d}x\}\over
\mmp\{|Q|>x\}}=\infty.$$ Thus, \eqref{negative2} holds.

To proceed we recall the already used notation $S_k:=\log |\Pi_k|$
and $\eta_{k+1}:=\log |Q_{k+1}|$, $k\in\mn_0$. According to
Corollary 4.19 (ii) in \cite{Resnick:1987} condition \eqref{20}
entails
\begin{equation}\label{70}
\sum_{k\geq 0}\1_{\{\eta_{k+1}>0\}}\varepsilon_{(n^{-1}k,\,
b_n^{-1}\eta_{k+1})} \ \Rightarrow \ N^{(1,\alpha)},  \ \
n\to\infty
\end{equation}
in $M_p$. If we can prove that
\begin{equation}\label{22}
{S_{[n\cdot]}\over b_n} \ \Rightarrow \ h(\cdot),\quad n\to\infty,
\end{equation}
where $h(t)=0$, $t\geq 0$, then relations \eqref{70} and
\eqref{22} can be combined into the joint convergence
\begin{equation*}
\bigg(b_n^{-1} S_{[n\cdot]},\sum_{k\geq
0}\1_{\{\eta_{k+1}>0\}}\varepsilon_{(n^{-1}k,\,
b_n^{-1}\eta_{k+1})}\bigg) \ \Rightarrow \ \big(h(\cdot),
N^{(1,\alpha)}\big), \ \ n\to\infty
\end{equation*}
in $D\times M_p$. By the Skorokhod representation theorem there
are versions which converge a.s. Retaining the original notation
for these versions we apply Proposition \ref{aux} with
$f_n(\cdot)=b_n^{-1}S_{[n\cdot]}$, $f_0=h$, $\nu_n=\sum_{k\geq 0}
\1_{\{\eta_{k+1}>0\}} \varepsilon_{\{n^{-1}k,\,
b_n^{-1}\eta_{k+1}\}}$, $\nu_0= N^{(1,\alpha)}$, $c_n=b_n$ and the
signs $\pm$ defined by ${\rm sgn}(\Pi_kQ_{k+1})$ which gives
\eqref{35} with $\log$ replaced with $\log^+$. The latter in
combination with \eqref{negative2} proves \eqref{35}.

It only remains to check \eqref{22}. To this end, it suffices to
prove that
\begin{equation}\label{34}{\underset{0\leq t\leq
T}{\sup}\,|S_{[nt]}|\over b_n}={\underset{0\leq k\leq
[nT]}{\max}\,|S_k|\over b_n} \ \overset{P}{\to} \ 0, \ \
n\to\infty
\end{equation}
for every $T>0$. Set
$$S_0^+=S_0^-:=0, \ S_n^+:=\log^+ |M_1|+\ldots+\log^+ |M_n|, \ S_n^-:=\log^-|M_1|+\ldots+\log^-|M_n|$$
for $n\in\mn$. Since $(b_n)$ is a regularly varying sequence and
$$\underset{0\leq k\leq [nT]}{\max}\,|S_k|\leq \underset{0\leq
k\leq [nT]}{\max}\,S^+_k+ \underset{0\leq k\leq
[nT]}{\max}\,S^-_k= S^+_{[nT]}+ S^-_{[nT]},$$ \eqref{34} follows
if we prove that $\lin (S^{\pm}_n/b_n)=0$ in probability. While
doing so, we treat two cases separately.

\noindent {\sc Case when $\me \log^- |M|<\infty$.} Then
necessarily $\me \log^+ |M|<\infty$ for otherwise $\lin
\Pi_n=\infty$ a.s. Therefore we have $\lin n^{-1}S^{\pm}_n=\me
\log^{\pm} |M|$ by the strong law of large numbers. Invoking
\eqref{19} proves \eqref{34}.

\noindent {\sc Case when $\me \log^- |M|=\infty$.} Condition
\eqref{18} entails $\lin {n\over b_n}\me \big((\log^-|M|)\wedge
b_n\big)=0$. Since
$${n\over b_n}\me \big((\log^-|M|)\wedge b_n\big)=n\mmp\{\log^-|M|>b_n\}+{n\over b_n}\me \log^-|M| \1_{\{\log^-|M|\leq b_n\}},$$ we infer
\begin{equation}\label{23}
\lin n\mmp\{\log^-|M|>b_n\}=0
\end{equation}
and
\begin{equation}\label{24}
\lin {n\over b_n}\me \big( \log^-|M| \1_{\{\log^-|M|\leq b_n\}}\big)=0.
\end{equation}
Using \eqref{24} together with Markov's inequality proves
$$\lin {\sum_{k=1}^n \log^-|M_k|\1_{\{\log^-|M_k|\leq b_n\}} \over b_n}=0 \ \ \text{in
probability}.$$ Since
\begin{eqnarray*}
\mmp\bigg\{b_n^{-1}\sum_{k=1}^n \log^-|M_k|\neq
b_n^{-1}\sum_{k=1}^n \log^-|M_k|\1_{\{\log^-|M_k|\leq
b_n\}}\bigg\}&\leq& \sum_{k=1}^n
\mmp\{\log^-|M_k|>b_n\}\\&=&n\mmp\{\log^-|M|>b_n\},
\end{eqnarray*}
\eqref{23} implies that the left-hand side tends to zero as
$n\to\infty$. Therefore $\lin (S^-_n/b_n)=0$ in probability.

Left
with proving that $\lin (S^+_n/b_n)=0$ in probability we suppose
immediately that $\me\log^+ |M|=\infty$ for the complementary case
can be treated in exactly the same way as above (use the strong
law of large numbers). Since $\lin S_n=-\infty$ a.s.\ by the
assumption, Lemma 8.1 in \cite{Pruitt:1981} tells us that $\lin
S_n^+/S_n^-=0$ a.s. which together with $\lin (S^-_n/b_n)=0$ in
probability implies $\lin (S^+_n/b_n)=0$ in probability. The proof
of \eqref{22} is complete. Hence so is that of \eqref{35}.

\noindent {\sc Proof of \eqref{355}} follows the pattern of that
of \eqref{32} but is simpler. Referring to \eqref{32} the only
things that need to be checked are that
\begin{equation*}
{\log |\Pi_{[n\cdot]+1}|\over b_n} \Rightarrow h(\cdot),\quad
n\to\infty,
\end{equation*}
where $h(t)=0$, $t\geq 0$, and that
$$\mmp\{\log |Q|-\log |M|>x\} \ \sim \ \mmp\{\log |Q|>x\} \ \sim \ x^{-\alpha}\ell(x), \ \
x\to\infty.$$ To prove the first of these, write
$${\underset{0\leq t\leq T}{\sup}\,|S_{[nt]+1}-S_{[nt]}|\over
b_n}\leq {\underset{0\leq t\leq T}{\sup}\,|S_{[nt]+1}|\over
b_n}+{\underset{0\leq t\leq T}{\sup}\,|S_{[nt]}|\over b_n}$$ and
use \eqref{34} to infer
$${\underset{0\leq t\leq T}{\sup}\,|S_{[nt]+1}-S_{[nt]}|\over
b_n} \ \overset{P}{\to} \ 0, \ \ n \to\infty$$ for every $T>0$. To
check the second we shall use \eqref{41}.

\noindent {\sc Case $\me\log^-|M|<\infty$.} We have $\lix
x\mmp\{\log^-|M|>\varepsilon x\}=0$ whereas $\lix x\mmp\{\log
|Q|>x\}=\infty$ (recall that in the case $\alpha=1$ we assume that
$\lix \ell(x)=\infty$). Therefore,
\begin{equation}\label{42}
\lix {\mmp\{\log^-|M|>\varepsilon x\}\over \mmp\{\log |Q|>x\}} =0.
\end{equation}
Since $\me\log^-|M|<\infty$ entails $\me\log^+|M|<\infty$, the
same argument proves \eqref{42} for the tail of $\log^+|M|$.

\noindent {\sc Case $\me\log^-|M|=\infty$ and
$\me\log^+|M|<\infty$}. It suffices to check \eqref{42} which is a
consequence \eqref{18}.

\noindent {\sc Case $\me\log^-|M|=\me\log^+|M|=\infty$}. We only
have to prove that $$\lix {\mmp\{\log^+|M|>\varepsilon x\}\over
\mmp\{\log |Q|>x\}}=0.$$ Since $\lin S_n=-\infty$ a.s.\ by the
assumption, we have $$\me {\log^+|M|\over A(\log^+|M|)}<\infty$$
(see Proposition 2.6 in \cite{Goldie+Maller:2000}). Therefore
$$\lix {x\mmp\{\log^+|M|>x\}\over \me (\log^-|M|\wedge x)}=0$$ and the
desired relation follows by an application of \eqref{18}.
\bigskip

\footnotesize \noindent   {\bf Acknowledgements}  \quad A.I.
thanks Alexander Marynych and Andrey Pilipenko for useful
discussions. D.B. was partially supported by the NCN grant
DEC-2012/05/B/ST1/00692.


\begin{thebibliography}{99}
\footnotesize

\bibitem{Alsmeyer+Iksanov+Roesler:2009} {\sc Alsmeyer, G., Iksanov, A. and R\"{o}sler,
U.} (2009). On distributional properties of perpetuities. {\em J.
Theor. Probab.} {\bf 22}, 666--682.

\bibitem{Babillot+Bougerol+Elie:1997} {\sc Babillot, M., Bougerol, Ph. and Elie,
L.} (1997). The random difference equation $X_n=A_nX_{n-1}+B_n$ in
the critical case. {\em Ann. Probab.} {\bf 25}, 478--493.

\bibitem{Basu+Roitershtein:2013} {\sc Basu, R. and Roiterstein,
A.} (2013). Divergent perpetuities modulated by regime switches.
{\em Stoch. Models.} {\bf 29}, 129--148.

\bibitem{Brofferio:2003} {\sc Brofferio, S.} (2003). How a centered random walk on the affine group goes to
infinity. {\em Ann. I. H. Poincar\'{e}}. {\bf 39}, 371–-384.

\bibitem{Brofferio+Buraczewski:2014+} {\sc Brofferio, S. and Buraczewski D.} (2014+).
On unbounded invariant measures of stochastic dynamical systems.
{\em Ann. Probab.}, to appear.


\bibitem{Buraczewski:2007} {\sc Buraczewski, D.} (2007). On invariant measures of stochastic recursions
in a critical case. {\em Ann. Appl. Probab.} {\bf 17}, 1245--1272.


\bibitem{Glynn+Whitt:1988} {\sc Glynn, P.~W. and Whitt, W.} (1988). Ordinary CLT and WLLN versions of $L=\lambda W$.
{\em Math. Oper. Res.} {\bf 13}, 674--692.

\bibitem{Goldie+Maller:2000} {\sc Goldie, C.~M. and Maller, R.~A.}
(2000). Stability of perpetuities. {\em Ann. Probab.} {\bf 28},
1195--1218.

\bibitem{Grincevicius:1975} {\sc Grincevicius, A.~K.} (1975).
Limit theorems for products of random linear transformations on
the line. {\em Lithuanian Math. J.} {\bf 15}, 568--579.

\bibitem{Hitczenko+Wesolowski:2011}{\sc Hitczenko, P. and Weso{\l}owski, J.} (2011). Renorming divergent
perpetuities. {\em Bernoulli}. {\bf 17}, 880--894.

\bibitem{Iksanov+Pilipenko:2014} {\sc Iksanov, A. and Pilipenko,
A.} (2014). On the maximum of a perturbed random walk. {\em Stat.
Probab. Letters}. {\bf 92}, 168--172.

\bibitem{Kellerer:1992} {\sc Kellerer, H.~ G.} (1992). Ergodic behaviour of affine recursions
I: criteria for recurrence and transience. Technical report,
University of Munich, Germany. Available at
http://www.mathematik.uni-muenchen.de/$\sim$ kellerer/

\bibitem{Pakes:1983} {\sc Pakes, A.~G.} (1983). Some properties of a random linear difference
equation. {\em Austral. J. Statist.} {\bf 25}, 345--357.

\bibitem{Pruitt:1981} {\sc Pruitt, W.~E.} (1981). General
one-sided laws of the iterated logarithm. {\em Ann. Probab.} {\bf
9}, 1--48.

\bibitem{Rachev+Samorodnitsky:1995} {\sc Rachev, S.~T. and Samorodnitsky, G.} (1995). Limit laws for a stochastic
process and random recursion arising in probabilistic modelling.
{\em Adv. Appl. Probab.} {\bf 27}, 185--202.

\bibitem{Resnick:1987} {\sc Resnick, S.} (1987). {\it Extreme values, regular variation, and point processes}. New York: Springer-Verlag.

\bibitem{Resnick:2007} {\sc Resnick, S.~I.} (2007). {\it Heavy-tail phenomena: Probabilistic and statistical
modeling}. New York: Springer.

\bibitem{Vervaat:1979}{\sc Vervaat, W.} (1979). On a stochastic difference equation and
a representation of non-negative infinitely divisible random
variables. {\em Adv. Appl. Probab.} {\bf 11}, 750--783.

\bibitem{Zeevi+Glynn:2004} {\sc Zeevi, A. and Glynn, P.~W.} (2004). Recurrence properties of autoregressive processes
with super-heavy-tailed innovations. {\em J. Appl. Probab.} {\bf
41}, 639--653.



\end{thebibliography}
\end{document}